\documentclass{amsart}
\usepackage[english]{babel}
\usepackage{amsfonts,bbm,mathrsfs,amssymb,amsmath,amsthm}

\newcommand{\tmop}[1]{\ensuremath{\operatorname{#1}}}
\newtheorem*{Mthm}{Main Theorem}
\newtheorem{thm}{Theorem}[section]
\newtheorem{theorem}{Theorem}[section]
\newtheorem{proposition}{Proposition}[section]
\newtheorem{lemma}{Lemma}[section]
\newtheorem{cor}{Corollary}[section]

\newtheorem{problem}{Problem}[section]
\newtheorem{defi}{Definition}[section]
\newtheorem*{cconj}{Chern Conjecture}
\newtheorem*{spp}{The Second Pinching Problem}
\theoremstyle{remark}
\newtheorem{rmk}{Remark}

\begin{document}

\title{On Chern's conjecture for minimal hypersurfaces in spheres}
\dedicatory{Dedicated to Professor Buqing Su on the occasion of his
115th birthday}
\thanks{Research supported by the National Natural Science Foundation of China, Grant Nos. 11531012, 11371315, 11601478; and
the China Postdoctoral Science Foundation, Grant No. 2016M590530.}

\author{Li Lei}
\address{Center of Mathematical Sciences \\ Zhejiang University \\ Hangzhou 310027 \\ China}
\email{lei-li@zju.edu.cn}

\author{Hongwei Xu}
\address{Center of Mathematical Sciences \\ Zhejiang University \\ Hangzhou 310027 \\China}
\email{xuhw@zju.edu.cn}

\author{Zhiyuan Xu}
\address{Center of Mathematical Sciences \\ Zhejiang University \\ Hangzhou 310027 \\ China}
\email{srxwing@zju.edu.cn}
\date{}
\keywords{Chern conjecture for minimal hypersurfaces; rigidity
theorem; scalar curvature; the second fundamental form}
\subjclass[2010]{53C24; 53C42}

\numberwithin{equation}{section}

\begin{abstract} Using a new estimate for the Peng-Terng invariant and the
multiple-parameter method, we verify a rigidity theorem on the
stronger version of Chern Conjecture for minimal hypersurfaces in
spheres. More precisely, we prove that if $M$ is a compact minimal
hypersurface in $\mathbb{S}^{n+1}$ whose squared length of the
second fundamental form satisfies $0\leq S-n\leq\frac{n}{18}$, then
$S\equiv n$ and $M$ is a Clifford torus.
\end{abstract}

\maketitle

\section{Introduction}
An important problem in global differential geometry is the study of
relations between geometrical invariants and structures of
Riemannian manifolds or submanifolds. In the late 1960's, Simons
\cite{S} , Lawson \cite{L1} , and Chern-do Carmo-Kobayashi
\cite{CDK} proved an optimal rigidity theorem for minimal
hypersurfaces in a sphere, which says that if $M$ is a compact
minimal hypersurface in the unit sphere $\mathbb{S}^{n+1}$, and if
the squared length of the second fundamental form of $M$ satisfies
$S\leq n$, then $S\equiv0$ and $M$ is the great sphere
$\mathbb{S}^{n}$, or $S\equiv n$ and $M$ is one of the Clifford
torus $\mathbb{S}^{k}\Big(\sqrt{\frac{k}{n}}\Big)\times
\mathbb{S}^{n-k}\Big(\sqrt{\frac{n-k}{n}}\Big)$, $\,1\le k\le n-1$.
More generally, they obtained a rigidity theorem for $n$-dimensional
compact minimal submanifolds in $\mathbb{S}^{n+p}$ under the
pinching condition $S\leq n/(2-\frac{1}{p})$. In \cite{LL}, Li-Li
improved Simons' pinching constant for $n$-dimensional compact
minimal submanifolds in $\mathbb{S}^{n+p}$ to $\frac{2}{3}n$ for the
case $p\geq3$. The famous Simons-Lawson-Chern-do
Carmo-Kobayashi-Li-Li rigidity theorem \cite{CDK,L1,LL,S} for
compact minimal submanifolds in a sphere is stated as follows.

\begin{theorem}
Let $M$ be an $n$-dimensional oriented compact minimal submanifold
in an $(n+p)$-dimensional unit sphere  $\mathbb{S}^{n+p}$. If the
squared length of the second fundamental form of $M$ satisfies
$S\leq \max\{\frac{n}{2-1/p},\frac{2}{3}n\}$. Then $M$ must be one
of the following:\\
(i) the great sphere $\mathbb{S}^{n}$ with $S\equiv0$;\\
(ii) the Clifford torus
$\mathbb{S}^{k}\Big(\sqrt{\frac{k}{n}}\Big)\times
\mathbb{S}^{n-k}\Big(\sqrt{\frac{n-k}{n}}\Big)$ with $S\equiv n$,
for $\,1\le k\le
n-1$;\\
(iii) the Veronese surface in $\mathbb{S}^{4}$ with $S\equiv
\frac{4}{3}$.
\end{theorem}

In 1975, Yau \cite{Y1} proved the rigidity theorem for compact
minimal submanifolds in a sphere under sectional curvature pinching
condition. In 1979, Ejiri \cite{Ej} obtained the rigidity theorem
for oriented compact simply connected minimal submanifolds with
pinched Ricci curvatures in a sphere. In 1984, Cheng-Li-Yau
\cite{CLY} proved a rigidity theorem for compact minimal
submanifolds in a sphere under volume pinching condition. In 1986,
Gauchman \cite{Ga} proved a rigidity theorem for compact minimal
submanifolds with bounded second fundamental form in a sphere.
Further developments on the rigidity theory of sphere type for
submanifolds have been made by many other authors, see for example
\cite{B1,B2,ChengNa,D,GXXZ,Lu,Shen,X0,X1,Y1}, etc.

The famous Chern Conjecture for minimal hypersurfaces in a sphere
was proposed by Chern \cite{Chern,CDK} in 1968 and 1970, and was
listed in the well-known problem section by Yau \cite{Y2} in 1982.
In \cite{Mun}, M\"unzner proved that if $M$ is a compact
isoparametric minimal hypersurface in $\mathbb{S}^{n+1}$, then
$g\in\{1,\,2,\,3,\,4,\,6\}$ and $S=(g-1)n$, where $g$ is the number
of distinct principal curvatures of $M$. In 1986,
Verstraelen-Montiel-Ros-Urbano gave the refined version of Chern
Conjecture. Afterwards, the second and third authors \cite{XX2,XX3}
formulated the stronger version of Chern Conjecture. The Chern
Conjecture for minimal hypersurfaces in spheres can be summarized as
follows.
\begin{cconj}\label{cj}
Let $M$ be a compact minimal hypersurface in the unit sphere $\mathbb{S}^{n+1}$.\\
{\rm (A)(Standard version)}  If $M$ has constant scalar curvature, then the possible values of the scalar curvature of $M$ form a discrete set. \\
{\rm (B)(Refined version)} If $M$ has constant scalar curvature, then $M$ is isoparametric. \\
{\rm (C)(Stronger version)} Denote by $S$ the squared length of the
second fundamental form of $M$. Set $a_k=(k-sgn(5-k))n$, for
$k\in \{m\in\mathbb{Z}^+; 1\leq m\leq 5\,\}$. Then we have\\
(i) For any fixed $k\in \{m\in\mathbb{Z}^+; 1\leq m\leq 4\,\}$, if
$a_k\leq S\leq a_{k+1}$, then $M$ is isoparametric, and
$S\equiv a_k$ or $S\equiv a_{k+1}$.\\
(ii) If $S\geq a_{5}$, then $M$ is isoparametric, and $S\equiv a_5$.
\end{cconj}

It's seen from the above that the Chern Conjecture consists of
several pinching problems. The first pinching problem for compact
minimal hypersurfaces in the unit sphere $\mathbb{S}^{n+1}$ was
solved by Simons, Lawson, and Chern-do Carmo-Kobayashi
\cite{CDK,L1,S}. The second pinching problem is an important part of
the Chern Conjecture, which has been open for almost fifty years.
\begin{spp}
Let $M$ be a compact minimal hypersurface in the unit sphere $\mathbb{S}^{n+1}$.\\
(i) If $S$ is constant, and if $n\leq S\leq 2n$, then $S=n$, or $S=2n$.\\
(ii) If $n\leq S\leq 2n$, then $S\equiv n$, or $S\equiv 2n$.
\end{spp}

In 1983, Peng-Terng \cite{PT1,PT2} initiated the study of the
second pinching problem for minimal hypersurfaces in the unit sphere,
and made the following breakthrough on the Chern Conjecture.
\begin{theorem}
Let $M$ be a compact minimal hypersurface in the unit sphere $\mathbb{S}^{n+1}$.\\
(i) If $S$ is constant, and if $n\leq S\leq n+\frac{1}{12n}$, then
$S=n$.\\
(ii) If $n\le 5$, and if
$n\leq S\leq n+\delta_1(n)$, where
$\delta_1(n)$ is a positive constant depending only on $n$, then
$S\equiv n$.
\end{theorem}

During the past three decades, there have been some important
progresses on Chern Conjecture (see \cite{Cheng,GeTang,GXXZ,SWY,Ver}
for more details). On the aspect of the standard version of Chern
Conjecture, Yang-Cheng \cite{YC0,YC1,YC2} improved the pinching
constant $\frac{1}{12n}$ in Theorem B(i) to $\frac{n}{3}$. Later,
Suh-Yang \cite{SY} improved this pinching constant to
$\frac{3}{7}n$.

In 1993, Chang \cite{C} proved Chern Conjecture (B) in dimension
three. Recently Deng-Gu-Wei \cite{DGW} showed that any closed
Willmore minimal hypersurface with constant scalar curvature in
$\mathbb{S}^{5}$ must be isoparametric.

On the stronger version of Chern Conjecture, Cheng-Ishikawa
\cite{ChengIshi} improved the pinching constant in Theorem B(ii) and
proved the second pinching theorem for compact minimal hypersurfaces
in $\mathbb{S}^{n+1}$ under the assumption that
$Ric_M\geq\frac{n}{2}$. In 2007, Wei-Xu \cite{WX} investigated Chern
Conjecture (C) and proved that if $M$ is a compact minimal
hypersurface in $\mathbb{S}^{n+1}$, $n=6, 7$, and if $n\leq S\leq n+
\delta_2(n)$, where $\delta_2(n)$ is a positive constant depending
only on $n$, then $S\equiv n$. Later, Zhang \cite{Zhang1} extended
the second pinching theorem due to Peng-Terng \cite{PT2} and Wei-Xu
\cite{WX} to the case of $n=8$. In 2011, Ding-Xin \cite{DX1}
obtained the following important rigidity result.
\begin{theorem}\label{thm2}
Let $M$ be an $n$-dimensional compact minimal hypersurface in the
unit sphere $\mathbb{S}^{n+1}$. If $n\geq6$, the squared norm of the
second fundamental form satisfies $0\leq S-n\leq
\frac{n}{23}$, then $S\equiv n$, i.e., $M$ is a Clifford torus.
\end{theorem}

In 2016, Xu-Xu \cite{XX3} gave a refined version of Ding-Xin's
rigidity theorem.
\begin{theorem}\label{thm3}
Let $M$ be  an $n$-dimensional compact minimal hypersurface in the
unit sphere $\mathbb{S}^{n+1}$. If the
squared length of the second fundamental form satisfies $0\leq S-n\leq\frac{n}{22}$, then $S\equiv n$ and $M$ is a Clifford torus.
\end{theorem}

In this paper, we verify the following rigidity theorem on the
stronger version of Chern Conjecture for minimal hypersurfaces in a
sphere.
\begin{Mthm}
Let $M$ be  an $n$-dimensional compact minimal hypersurface in the
unit sphere $\mathbb{S}^{n+1}$. If the
squared length of the second fundamental form satisfies $0\leq S-n\leq\frac{n}{18}$, then $S\equiv n$ and $M$ is a Clifford torus.
\end{Mthm}

The key ingredients of the proof of Main Theorem are as follows.
With the aid of the Sylvester theory, we derive a good estimate on
the upper bound for the Peng-Terng invariant $A-2B$. Using this
estimate and the multiple-parameter method (i.e., the generalized
Yau parameter method), we establish a new integral inequality on the
second fundamental form of $M$ and its gradient. By choosing
appropriate parameters, we verify that the coefficients of the
integrals in (\ref{ineq3p}) are both negative.

\section{Minimal hypersurfaces in a sphere}

Let $M$ be a hypersurface in the unit sphere $\mathbb{S}^{n + 1}$.
Denote by $\overline{\nabla}$ and $\nabla$ the Levi-Civita
connection on $\mathbb{S}^{n + 1}$ and the induced connection on
$M$, respectively. Let $h$ be the second fundamental form of $M$.
For tangent vector fields $X$ and $Y$ over $M$, we have the Gauss
formula
\[ \overline{\nabla}_X Y = \nabla_X Y + h (X, Y) . \]
We shall make use of the following convention on the range of
indices:
$$1\leq i, j, k, \ldots\leq n.$$
Choose a local orthonormal frame $\{ e_i \}$ for the tangent bundle
over $M$. Let $\nu$ be a local unit normal vector field of $M$. Set
$h (e_i, e_j) = h_{i j} \nu$. Denote by $S$ the squared length of
the second fundamental form of $M$, i.e., $S = \sum_{i, j} h_{i
j}^2$. We denote by $h_{i j k}$ the covariant derivative of $h_{i
j}$. It follows from the Codazzi equation that $h_{i j k}$ is
symmetric in $i, j$ and $k$.

From now on, we assume that $M$ is a compact minimal hypersurface in
$\mathbb{S}^{n + 1}$. Then $\sum_i h_{i i} = 0$. The following
Simons formula and Peng-Terng formula for minimal hypersurfaces in
$\mathbb{S}^{n + 1}$ can be found in \cite{S} and \cite{PT1,PT2},
respectively.
\begin{equation}
  \frac{1}{2} \Delta S = | \nabla h |^2 + S (n - S), \label{LapS}
\end{equation}

\begin{equation}
  \frac{1}{2} \Delta | \nabla h |^2 = | \nabla^2 h |^2 + (2 n + 3 - S) |
  \nabla h |^2 - 3 (A - 2 B) - \frac{3}{2} | \nabla S |^2, \label{Lapdh2}
\end{equation}
where
\[ A = \sum_{i, j, k, l, m} h_{i j k} h_{i j l} h_{k m} h_{m l}, \qquad B =
   \sum_{i, j, k, l, m} h_{i j k} h_{k l m} h_{i m} h_{j l} . \]
\begin{defi}
$A-2B$ is called the Peng-Terng invariant of $M$.
\end{defi}
Integrating (\ref{LapS}) and (\ref{Lapdh2}), we get
\begin{equation}
 \int_M | \nabla h |^2 d M= \int_M S (S - n) d M
\end{equation}
and
\begin{equation}
  \int_M | \nabla^2 h |^2 d M= \int_M \left[ - (2 n + 3 - S) | \nabla h |^2 + 3
  (A - 2 B) + \frac{3}{2} | \nabla S |^2 \right]d M. \label{ddh2}
\end{equation}

Choose a local orthonormal frame such that $h_{i j} = \lambda_i \delta_{i
j}$. Then $\sum_i \lambda_i = 0$, $\sum_i \lambda_i^2 = S$ and
\[ A - 2 B = \sum_{i, j, k} h_{i j k}^2 (\lambda_i^2 - 2 \lambda_i \lambda_j)
   . \]
For a positive integer $m$, we put $f_m = \sum_i \lambda_i^m$.
Peng-Terng \cite{PT2} obtained the following integral formula.
\[ \int_M (A - 2 B) d M= \int_M \left( S f_4 - f_3^2 - S^2 - \frac{1}{4} | \nabla
   S |^2 \right) d M. \]
Define
\[ F = \sum_{i, j} (\lambda_i -
   \lambda_j)^2 (1 + \lambda_i \lambda_j)^2 = 2 [S f_4 - f_3^2 - S^2 - S (S - n)]. \]
Then we get
\begin{equation}
  \int_M (A - 2 B) d M= \int_M \left( \frac{1}{2} F + | \nabla h |^2 -
  \frac{1}{4} | \nabla S |^2 \right) d M. \label{int2B}
\end{equation}
It follows from (\ref{LapS}) that
\[ \frac{1}{4} \Delta S^2 = \frac{1}{2} | \nabla S |^2 + S | \nabla h |^2 +
   S^2 (n - S) . \]
This implies
\begin{equation}
  \frac{1}{2} \int_M | \nabla S |^2 d M= \int_M [(n - S) | \nabla h |^2 + S (S -
  n)^2] d M. \label{intdS2}
\end{equation}
In \cite{ChengIshi}, Cheng and  Ishikawa gave the following
estimate.
\begin{equation}
  | \nabla^2 h |^2 \geq \frac{3}{4} F + \frac{3 S (S - n)^2}{2 (n + 4)} .
  \label{ddh2greaG}
\end{equation}

\section{An upper bound for the Peng-Terng invariant}

In this section, we will derive a new estimate on the upper bound
for the Peng-Terng invariant of minimal hypersurfaces in the unit
sphere $\mathbb{S}^{n+1}$, which is more effective than one given by
Ding-Xin \cite{DX1} in the study of the second pinching problem for
minimal hypersurfaces.

\begin{thm}\label{3Bless}
Let $M$ be a minimal hypersurface in $\mathbb{S}^{n + 1}$. If $n
\geq 6$, $\eta > 0$, and if $n \leq S \leq
  (1 + \eta^{- 1}) n$, then
  \[ 3 (A - 2 B) \leq \left( S + 4 + \sqrt[3]{c F} \right) | \nabla h
     |^2, \]
  where $c = \frac{24}{5} - \frac{16}{(1 + \eta^{- 1}) n}$.
\end{thm}

To prove Theorem \ref{3Bless}, we need establish several crucial
algebraic inequalities. For this purpose, we recall the Sylvester
theory for resultants of polynomials. Let $P$ and $Q$ be two nonzero
polynomials in the indeterminate $x$, respectively of degree $m$ and
$k$. More precisely, let
\[ P = p_m x^m + p_{m - 1} x^{m - 1} + \cdots + p_1 x + p_0, \]
\[ Q = q_k x^k + q_{k - 1} x^{k - 1} + \cdots + q_1 x + q_0 . \]
The Sylvester matrix of $P$ and $Q$ is an $(m + k) \times (m + k)$ matrix
formed by
\[ \tmop{Syl}_x (P, Q) = \left[\begin{array}{ccccccc}
     p_m & p_{m - 1} & \cdots & p_0 &  &  & \\
     & p_m & p_{m - 1} & \cdots & p_0 &  & \\
     &  & \ddots \quad & \ddots &  & \ddots & \\
     &  &  & p_m & p_{m - 1} & \cdots & p_0\\
     q_k & q_{k - 1} & \cdots & q_0 &  &  & \\
     & q_k & q_{k - 1} & \cdots & q_0 &  & \\
     &  & \ddots \quad & \ddots &  & \ddots & \\
     &  &  & q_k & q_{k - 1} & \cdots & q_0
   \end{array}\right]
   \!\!\!\!\!\!\!\!
   \begin{array}{l}
     \left. \begin{array}{l}
       \\ \\ \\ \\
     \end{array} \right\} k \enspace \tmop{rows}\\ \\
     \left. \begin{array}{l}
       \\ \\ \\ \\
     \end{array} \right\} m \enspace \tmop{rows}
   \end{array} . \]
The resultant of $P$ and $Q$ is defined as the determinant of their Sylvester
matrix:
\[ \tmop{res}_x (P, Q) = \det (\tmop{Syl}_x (P, Q)) . \]
The resultant of two polynomials with complex coefficients is zero if and only
if they have a common root in the complex number field.

Let $P$ be a polynomial in the indeterminate $x$ of positive degree $m$, with
$p_m$ as leading coefficient. Then its derivative $P'$ is a polynomial of
degree $m - 1$. The discriminant of $P$ is defined as
\[ \tmop{disc}_x (P) = \frac{(- 1)^{\frac{m (m - 1)}{2}}}{p_m} \tmop{res}_x
   (P, P') . \]

\begin{proposition}\label{prop3.1} For a polynomial with real coefficients, the sign of the discriminant provides
the following information on the nature of the roots of the polynomial.\\
(i) The discriminant is zero if and only if the polynomial has a
multiple root. \\
(ii) The discriminant is positive if and only if the number of
non-real roots is a
multiple of 4. \\
(iii) The discriminant is negative if and only if the number of
non-real roots is even but not a multiple of 4.
\end{proposition}

The following lemmas will be used in the proof of Theorem
\ref{3Bless}.

\begin{lemma}
  \label{ineqef}For all $\sigma, \tau \in \mathbb{R}$, we have
  \begin{eqnarray}
    &  & 6 \tau^6 + 4 (4 \sigma^2 + 9) \tau^4 - 2 (10 \sigma^4 + 108 \sigma^2 +
    243) \tau^2 \nonumber\\
    &  & + 7 \sigma^6 + 126 \sigma^4 + 729 \sigma^2 + 1296 \label{sigmatau}\\
    & > & (| \sigma^3 - 2 \sigma \tau^2 + 9 \sigma | + 2 | \tau |^3)^2 . \nonumber
  \end{eqnarray}
\end{lemma}

\begin{proof}
  We set
  \begin{eqnarray*}
    P (\sigma, \tau) & = & \frac{1}{2} [6 \tau^6 + 4 (4 \sigma^2 + 9) \tau^4 - 2
    (10 \sigma^4 + 108 \sigma^2 + 243) \tau^2 \\
    &  &  + 7 \sigma^6 + 126 \sigma^4 + 729 \sigma^2 + 1296 -
    (\sigma^3 - 2 \sigma \tau^2 + 9 \sigma + 2 \tau^3)^2] .
  \end{eqnarray*}
Denote by $LHS$ and $RHS$ the left hand side and the right hand side
of
(\ref{sigmatau}), respectively. We see the following facts.\\
(i) If $\sigma^3 - 2 \sigma \tau^2 + 9 \sigma$ and $\tau$ have the
same sign, then $LHS-RHS=2 P (\sigma, \tau)$.\\
(ii) If $\sigma^3 - 2 \sigma
  \tau^2 + 9 \sigma$ and $\tau$ have different sign, then $LHS-RHS=2 P (\sigma, -\tau).$

Thus, to prove Lemma \ref{ineqef}, it is sufficient to show that
  $P (\sigma,  \tau) > 0$ for all $\sigma, \tau \in \mathbb{R}$.

  Computing the discriminant of $P$ with respect to $\tau$, we get
  \begin{eqnarray*}
     \tmop{disc}_{\tau} (P (\sigma, \tau))
    & = & - 102036672 (\sigma^2 + 6)^3 \times\\
    &  & (12 \,\sigma^{12} + 508 \,\sigma^{10} + 9034 \,\sigma^8 + 86582 \,\sigma^6 \\&&+
    471177 \,\sigma^4 + 1376352 \,\sigma^2 + 1679616)\\
    & < & 0.
  \end{eqnarray*}

  Set $\Psi (\sigma) = \inf_{\tau \in \mathbb{R}} P (\sigma, \tau)$. Since
  $\tmop{disc}_{\tau} (P (\sigma, \tau))$ can never be zero, we get $\Psi
  (\sigma) \neq 0$ for all $\sigma \in \mathbb{R}$. On the other hand, we
  have
  \begin{eqnarray*}
    P (0, \tau) & = & \tau^6 + 18 \tau^4 - 243 \tau^2 + 648\\
    & = & \left( \tau^2 + 6 \sqrt{13} + 6 \right)  \left( \tau^2 + 6 - 3
    \sqrt{13} \right)^2 + \frac{648}{13 \sqrt{13} + 47}\\
    & \geq & \frac{648}{13 \sqrt{13} + 47} .
  \end{eqnarray*}
  From $\Psi (0) > 0$ and the continuity of $\Psi (\sigma)$, we obtain $\Psi
  (\sigma) > 0$ for all $\sigma \in \mathbb{R}$. Hence we have proved $P
  (\sigma, \tau) > 0$ for all $\sigma, \tau \in \mathbb{R}$.
\end{proof}

\begin{lemma}
  \label{AlIneq2}For all $x, y, z \in \mathbb{R}$, we have
  \begin{eqnarray*}
   &&- 2 (x y + y z + z x + 2)^3 \\&<& (x - y)^2  (xy + 1)^2 + (x - z)^2  (xz +
     1)^2 + (z - y)^2  (yz + 1)^2 .
  \end{eqnarray*}
\end{lemma}

\begin{proof}
  Set
  \begin{eqnarray*}
    P &=& 2 (x y + y z + z x + 2)^3 \\&&+ (x - y)^2  (xy + 1)^2 + (x - z)^2  (xz +
     1)^2 + (z - y)^2  (yz + 1)^2 .
        \end{eqnarray*}
  Since $P$ is a symmetric polynomial, it can be expressed in terms of
  elementary symmetric polynomials. Let
  \[ \sigma_1 = x + y + z, \quad \sigma_2 = x y + y z + z x, \quad \sigma_3 =
     x y z, \]
  \[ \tau = \sqrt{\frac{1}{2} [(x - y)^2 + (x - z)^2 + (z - y)^2 ]} =
     \sqrt{\sigma_1^2 - 3 \sigma_2} . \]
  Then we have
  \begin{eqnarray*}
    P & = & - 9 \sigma_3^2 + (- 2 \sigma_1^3 + 10 \sigma_2 \sigma_1 + 6
    \sigma_1) \sigma_3\\
    &  & - 2 \sigma_2^3 + \sigma_1^2 \sigma_2^2 + 4 \sigma_2^2 + 2 \sigma_1^2
    \sigma_2 + 18 \sigma_2 + 2 \sigma_1^2 + 16\\
    & = & \frac{1}{81} \times [- (- 5 \sigma_1 \tau^2 + 2 \sigma_1^3 + 9
    \sigma_1 - 27 \sigma_3)^2 \\
    &  & + 6 \tau^6 + 4 (4 \sigma_1^2 + 9) \tau^4 - 2 (10 \sigma_1^4 + 108
    \sigma_1^2 + 243) \tau^2\\
    &  &  + 7 \sigma_1^6 + 126 \sigma_1^4 + 729 \sigma_1^2 + 1296]
    .
  \end{eqnarray*}
  Using Lemma \ref{ineqef}, we get
  \begin{equation}
    81 P > (| \sigma_1^3 - 2 \sigma_1 \tau^2 + 9 \sigma_1 | + 2 \tau^3)^2 - (-
    5 \sigma_1 \tau^2 + 2 \sigma_1^3 + 9 \sigma_1 - 27 \sigma_3)^2 .
    \label{81Pgre}
  \end{equation}

  Notice that $x, y, z$ are the three real roots of the polynomials $\lambda^3 -
  \sigma_1 \lambda^2 + \sigma_2 \lambda - \sigma_3$. We have
  \begin{eqnarray*}
    0 & \leq & \tmop{disc}_{\lambda} (\lambda^3 - \sigma_1 \lambda^2 +
    \sigma_2 \lambda - \sigma_3)\\
    & = & \sigma_1^2 \sigma_2^2 - 4 \sigma_2^3 - 4 \sigma_1^3 \sigma_3 - 27
    \sigma_3^2 + 18 \sigma_1 \sigma_2 \sigma_3\\
    & = & \frac{4}{27} \tau^6 - \frac{1}{27}  (3 \sigma_1 \tau^2 - \sigma_1^3 +
    27 \sigma_3)^2 .
  \end{eqnarray*}
  Thus we have
  \begin{equation}
    2 \tau^3 \geq | 3 \sigma_1 \tau^2 - \sigma_1^3 + 27 \sigma_3 | .
    \label{2f3gre}
  \end{equation}

  From (\ref{81Pgre}) and (\ref{2f3gre}), we obtain
  \begin{eqnarray*}
    81 P & > & (| \sigma_1^3 - 2 \sigma_1 \tau^2 + 9 \sigma_1 | + | 3 \sigma_1
    \tau^2 - \sigma_1^3 + 27 \sigma_3 |)^2\\
    &  & - (- 5 \sigma_1 \tau^2 + 2 \sigma_1^3 + 9 \sigma_1 - 27
    \sigma_3)^2\\
    & \geq & [(\sigma_1^3 - 2 \sigma_1 \tau^2 + 9 \sigma_1) - (3
    \sigma_1 \tau^2 - \sigma_1^3 + 27 \sigma_3)]^2\\
    &  & - (- 5 \sigma_1 \tau^2 + 2 \sigma_1^3 + 9 \sigma_1 - 27
    \sigma_3)^2\\
    & = & 0.
  \end{eqnarray*}
\end{proof}

\begin{lemma}
  \label{ineqks}For $s \geq 6$, $k \in \mathbb{R}$, we have
  \[ 16 (3 s - 10) (k - 1)^2 (k^2 + ks + 1)^2 + 5 (4 k^2 + 4 ks + s + 4)^3 >
     0. \]
\end{lemma}

\begin{proof}
    Putting $s = r + 6$, we have
    \begin{eqnarray*}
  &  & 16 (3 s - 10) (k - 1)^2 (k^2 + ks + 1)^2 + 5 (4 k^2 + 4 ks + s +
  4)^3\\
  & = & 16 (3 r + 8)  (k - 1)^2 [k^2 + k (r + 6) + 1]^2 + 5 [4 k^2 + 4 k (r +
  6) + r + 10]^3 .
    \end{eqnarray*}
    Set
    \[ P (r, k) = 16 (3 r + 8)  (k - 1)^2 [k^2 + k (r + 6) + 1]^2 + 5 [4 k^2 + 4 k
   (r + 6) + r + 10]^3 . \]
    We need to prove $P (r, k) > 0$ for all $r \geq 0$ and $k \in
\mathbbm{R}$. Computing the discriminant of $P$ with respect to $k$, we get
    \begin{eqnarray*}
        \tmop{disc}_k (P (r, k)) & = & - 8388608000 (r + 6)^7  (3 r + 8)^4  (5 r + 38)^3
        \times\\
        &  & (432 \,r^{10} + 85536 \,r^9 + 3803796 \,r^8 \\
        &  & + 82050188 \,r^7 + 1045887247 \,r^6\\
        &  & + 8514043782 \,r^5 + 45438798848 \,r^4\\
        &  & + 157585300528 \,r^3 + 338704428144 \,r^2\\
        &  &  + 402431922656 \,r + 195043474048)\\
        & < & 0.
    \end{eqnarray*}
    Set $\Psi (r) = \inf_{k \in \mathbb{R}} P (r, k)$. For any $r \geq 0$, from $\tmop{disc}_k (P
    (r, k))\neq 0$ we get $\Psi (r) \neq 0$.

    Now we estimate $\Psi (0)$. We have
    \begin{eqnarray*}
        P (0, k) & = & 32 \times [14 k^6 + 220 k^5 + 1215 k^4 + 2852 k^3
        \\
        &  & \left. + \left( 2947 \!+\! \tfrac{1}{2} \right) k^2 + 1165 k +
        160 \!+\! \tfrac{1}{4} \right] .
    \end{eqnarray*}
    Let
    \[ Q (k) = 14 k^6 + 220 k^5 + 1215 k^4 + 2852 k^3 + 2947 k^2 + 1165 k +
    160. \]
    Then $P (0, k) > 32 Q (k)$.

    Let $l$ be a nonnegative number. A computation gives
    \begin{eqnarray*}
  &&\tmop{disc}_k (Q (k) + l) \\& = & - 8 \times (3136589568 \,l^5 + 11043385174784
  \,l^4 \\
  &  & + 1758965584701728 \,l^3 + 79189061386916048 \,l^2\\
  &  &  + 1067453304129927340 \,l + 4262062225186419475)\\
  & < & 0.
    \end{eqnarray*}
    Hence, we get $\inf_{k \in
\mathbb{R}} (Q (k) + l) \neq 0$ for all $l \geq 0$. It follows that
$\inf_{k \in \mathbb{R}} Q (k) > 0$.
    Thus we have $\Psi (0) \geq 32 \inf_{k \in \mathbb{R}}Q (k) > 0$.
        Since $\Psi (r)$ is continuous, we obtain $\Psi (r) > 0$ for $r \geq 0$.
\end{proof}

\begin{lemma}
  \label{AlIneq1}Let $s$ be a positive number satisfying $s \geq 6$, let
  $D_s = \{ (x, y) ; x^2 + y^2 \leq s \}$. For any point $(x, y) \in D_s$, we
  have
  \[ - (x^2 + 4 x y + 4)^3 < \frac{16}{5}  \left( 3 - \frac{10}{s} \right) (x
     - y)^2 (1 + x y)^2 . \]
\end{lemma}

\begin{proof}
  If $(x - y) (1 + x y) = 0$ or $x^2 + 4 x y + 4 = 0$, the assertion is
  obviously true.

  Suppose that $(x - y) (1 + x y) \neq 0$ and $x^2 + 4 x y + 4 \neq 0$. Set
  \[ \varphi = - \frac{(x^2 + 4 x y + 4)^3}{(x - y)^2 (1 + x y)^2} . \]
  Now we find the critical points of $\varphi$. The partial derivatives of $\varphi$
  are given by
  \[ \frac{\partial \varphi}{\partial x} = - \frac{2 (x^2 + 4 xy + 4)^2  (x^3
     y - 4 x^2 y^2 + 2 x^2 - 2 xy^3 - 9 xy - 2 y^2 - 4)}{(x - y)^3  (xy +
     1)^3}, \]
  \[ \frac{\partial \varphi}{\partial y} = \frac{2 (x^2 + 4 xy + 4)^2  (x^4 -
     4 x^3 y - 2 x^2 y^2 - 3 x^2 - 6 xy - 4)}{(x - y)^3  (xy + 1)^3} . \]
  We investigate the simultaneous equations
  \begin{equation}
    \left\{ \begin{array}{rll}
      x^3 y - 4 x^2 y^2 + 2 x^2 - 2 xy^3 - 9 xy - 2 y^2 - 4 & = & 0,\\
      x^4 - 4 x^3 y - 2 x^2 y^2 - 3 x^2 - 6 xy - 4 & = & 0.
    \end{array} \right. \label{eqs}
  \end{equation}
  Computing the resultant with respect to $x$, we get
  \begin{eqnarray*}
    &  & \tmop{res}_x (x^3 y - 4 x^2 y^2 + 2 x^2 - 2 xy^3 - 9 xy - 2 y^2 - 4,
    \\
    &  & \qquad  x^4 - 4 x^3 y - 2 x^2 y^2 - 3 x^2 - 6 xy - 4)\\
    & = & 720 y^4 + 1296 y^2 + 576.
  \end{eqnarray*}
  Since the resultant can not be zero for real $y$, the equation system
  (\ref{eqs}) has no real solutions. Hence, the maximum of $\varphi$ on $D_s$
  is achieved on the boundary $\partial D_s$.

  Now we estimate the maximum of $\varphi$ on $\partial D_s = \{ (x, y) : x^2
  + y^2 = s \}$. If $x = 0$, the assertion is obviously true. Assume $x \neq 0$.
  Let $y = k x$. Then $x^2 = \frac{s}{1 + k^2}$. Thus we obtain
  \[ \varphi = - \frac{(4 k^2 + 4 ks + s + 4)^3}{s (k - 1)^2 (k^2 + ks + 1)^2}
     . \]
  Using Lemma \ref{ineqks}, we have
  \begin{eqnarray*}
    &  & \varphi - \frac{16}{5}  \left( 3 - \frac{10}{s} \right)\\
    & = & - \frac{1}{5 s (k - 1)^2 (k^2 + ks + 1)^2} \times\\
    &  & [16 (3 s - 10) (k - 1)^2 (k^2 + ks + 1)^2 + 5 (4 k^2 + 4 ks + s +
    4)^3]\\
    & < & 0.
  \end{eqnarray*}
\end{proof}

With the aid of the above lemmas, we can now prove Theorem
\ref{3Bless}.
\begin{proof}[Proof of Theorem \ref{3Bless}.]
  From the definitions of $A$ and $B$, we have
  \begin{eqnarray*}
    3 (A - 2 B) & = & \sum_{i, j, k \tmop{distinct}} h_{i j k}^2 (\lambda_i^2
    + \lambda_j^2 + \lambda_k^2 - 2 \lambda_i \lambda_j - 2 \lambda_j
    \lambda_k - 2 \lambda_i \lambda_k)\\
    &  & + 3 \sum_{i, j \tmop{distinct}} h_{i i j}^2 (\lambda_j^2 - 4
    \lambda_i \lambda_j) - 3 \sum_i h_{i i i}^2 \lambda_i^2 .
  \end{eqnarray*}
  By the definition of $F$, for distinct $i, j, k$, we get
  \begin{eqnarray*}
    F & \geq & 2 (\lambda_i - \lambda_j)^2 (1 + \lambda_i \lambda_j)^2 +
    2 (\lambda_j - \lambda_k)^2 (1 + \lambda_j \lambda_k)^2 + 2 (\lambda_i -
    \lambda_k)^2 (1 + \lambda_i \lambda_k)^2\\
    & \geq & 2 (\lambda_i - \lambda_j)^2 (1 + \lambda_i \lambda_j)^2 .
  \end{eqnarray*}
  Applying Lemma \ref{AlIneq2}, we obtain
  \begin{eqnarray*}
    &  & - 2 \lambda_i \lambda_j - 2 \lambda_j \lambda_k - 2 \lambda_i
    \lambda_k - 4\\
    & < & [4 (\lambda_i - \lambda_j)^2 (1 + \lambda_i \lambda_j)^2 + 4
    (\lambda_j - \lambda_k)^2 (1 + \lambda_j \lambda_k)^2 + 4 (\lambda_i -
    \lambda_k)^2 (1 + \lambda_i \lambda_k)^2]^{\frac{1}{3}}\\
    & \leq & (2 F)^{\frac{1}{3}} .
  \end{eqnarray*}
  Noting that $c > \frac{24}{5} - \frac{16}{6} > 2$, we get
  \[ \lambda_i^2 + \lambda_j^2 + \lambda_k^2 - 2 \lambda_i \lambda_j - 2
     \lambda_j \lambda_k - 2 \lambda_i \lambda_k \leq S + 4 + \sqrt[3]{c F} . \]
  By Lemma \ref{AlIneq1}, we have
  \begin{eqnarray*}
    \lambda_j^2 - 4 \lambda_i \lambda_j & = & \lambda_i^2 + \lambda_j^2 + 4 -
    (4 \lambda_i \lambda_j + \lambda_i^2 + 4)\\
    & < & S + 4 + \left[ \frac{16}{5}  \left( 3 - \frac{10}{S} \right) (\lambda_i - \lambda_j)^2
    (1 + \lambda_i \lambda_j)^2 \right]^{\frac{1}{3}}\\
    & \leq & S + 4 + \left[ \left( \frac{24}{5} - \frac{16}{S} \right) F
    \right]^{\frac{1}{3}}\\
    & \leq & S + 4 + \sqrt[3]{c F} .
  \end{eqnarray*}
  Therefore, we obtain
  \begin{eqnarray*}
    3 (A - 2 B) & \leq & \left( S + 4 + \sqrt[3]{c F} \right) \left(
    \sum_{i, j, k \tmop{distinct}} h_{i j k}^2 + 3 \sum_{i, j \tmop{distinct}}
    h_{i i j}^2 \right)\\
    & \leq & \left( S + 4 + \sqrt[3]{c F} \right) | \nabla h |^2 .
  \end{eqnarray*}
\end{proof}

\begin{cor}
\label{BB} Let $n \geq 6$. If $n \leq S \leq 2n$,
then
  \[ 3 (A - 2 B) \leq \left( S + 4 + \sqrt[3]{c F} \right) | \nabla h
     |^2, \]
  where $c = \frac{24}{5} - \frac{8}{n}$.
\end{cor}

\begin{rmk}
It seems that the crucial point to solve the second pinching problem
on the Chern conjecture is to find the best constant $c$ in
Corollary \ref{BB}.
\end{rmk}

\section{Proof of the main theorem}

Now we are in a position to prove our rigidity theorem on the Chern
Conjecture for  minimal hypersurfaces in a sphere.

\begin{proof}[Proof of Main Theorem]
When $n = 3$, Peng-Terng \cite{PT2} proved that if $0\leq
S-3\leq\delta(3)$, then $S\equiv3$, where
$\delta(3)=\frac{2.61}{5+\sqrt{17}}>\frac{1}{6}$. When $n = 4$,
Cheng-Ishikawa \cite{ChengIshi} verified that if $0\leq
S-4\leq\delta(4)$, then $S\equiv4$, where
$\delta(4)=\frac{8}{31}>\frac{2}{9}$. When $n = 5$, Zhang
\cite{Zhang1} showed that if $0\leq S-5\leq\delta(5)$, then $S\equiv
5$, where $\delta(5)=0.444>\frac{5}{18}$.

When $n \geq 6$, we assume that $M$ is a compact minimal
hypersurface in $\mathbb{S}^{n+1}$ whose squared length of the
second fundamental form satisfies  $n \leq S \leq (1 +
\eta^{- 1}) n$, where $\eta$ is a positive parameter.

Combining (\ref{ddh2}), (\ref{int2B}) and (\ref{ddh2greaG}), we have
\begin{eqnarray}
  \int_M (A - 2 B)d M
  & = & \nonumber\int_M \left[ \frac{1}{2} F + | \nabla h |^2 - \frac{1}{4} | \nabla S
  |^2 \right] d M\\
  & \leq &\nonumber \int_M \left[ \frac{2}{3} | \nabla^2 h |^2 - \frac{S (S -
  n)^2}{n + 4} + | \nabla h |^2 - \frac{1}{4} | \nabla S |^2 \right]d M
  \label{intBeq}\\
  & = & \int_M \bigg[ - \bigg( \frac{4}{3} n + 1 - \frac{2}{3} S \bigg) |
  \nabla h |^2 + 2 (A - 2 B)\\
  & &\nonumber - \frac{S (S - n)^2}{n + 4} + \frac{3}{4} |
  \nabla S |^2 \bigg] d M.
\end{eqnarray}
By (\ref{intdS2}), we get
\begin{equation}
   \label{dS2SS-2n}
   \int_M \left[ \frac{3}{4} | \nabla S |^2 - \frac{S (S - n)^2}{n + 4}
  \right] d M
   =  \int_M \left[ \frac{3}{2} (n - S) | \nabla h |^2 + b S (S - n)^2
  \right] d M ,
\end{equation}
where $b = \frac{3}{2} - \frac{1}{n + 4}$. Since $S \leq (1 +
\eta^{- 1}) n$, we have
\begin{equation}
   \label{dS2SS-3n} \int_M   b S (S - n)^2
   d M \leq  \int_M   \frac{b
  n}{\eta} S (S - n)  d M =\int_M  \frac{b n}{\eta}  | \nabla
  h |^2 d M.\end{equation}
Substituting (\ref{dS2SS-3n}) into (\ref{dS2SS-2n}), we get
\begin{equation}
   \label{dS2SS-n}
    \int_M \left[ \frac{3}{4} | \nabla S |^2 - \frac{S (S - n)^2}{n + 4}
  \right] d M
   \leq  \int_M \left[ \frac{3}{2} (n - S) + \frac{b n}{\eta} \right] | \nabla
  h |^2 d M.
\end{equation}
From (\ref{intBeq}) and (\ref{dS2SS-n}), we obtain
\begin{equation}
  \int_M (A - 2 B) d M\geq \int_M \left( 1 - \frac{n}{6} + \frac{5}{6} S -
  \frac{b n}{\eta} \right) | \nabla h |^2 d M. \label{intBgre}
\end{equation}

Let $\sigma$ be a positive parameter. Using Theorem \ref{3Bless} and
Young's inequality, we get
\begin{eqnarray}
  3 (A - 2 B) & \leq & \left( S + 4 + \sqrt[3]{c F} \right) | \nabla h
  |^2  \label{3Bsigma}\\
  & \leq & (S + 4) | \nabla h |^2 + \frac{1}{3} c \sigma^2 F +
  \frac{2}{3 \sigma} | \nabla h |^3 . \nonumber
\end{eqnarray}
Let $\varepsilon$ and $\kappa$ be positive parameters. We have the
following estimates.
\begin{eqnarray}
  \int_M | \nabla h |^3 d M& = & \int_M \left[ S (S - n) | \nabla h | +
  \frac{1}{2} | \nabla h | \Delta S \right] d M \nonumber\\
  & = & \int_M \left[ S (S - n) | \nabla h | - \frac{1}{2} \langle \nabla |
  \nabla h |, \nabla S \rangle \right] d M \label{dh3}\\
  & \leq & \int_M \left[ S (S - n) | \nabla h | + \varepsilon | \nabla^2
  h |^2 + \frac{1}{16 \varepsilon} | \nabla S |^2 \right] d M \nonumber
\end{eqnarray}
and
\begin{eqnarray}
  &  & \int_M S (S - n) | \nabla h | d M \nonumber\\
  & \leq & \int_M \left[ 2 (1 + \eta^{- 1}) n \kappa S (S - n) +
  \frac{1}{8 (1 + \eta^{- 1}) n \kappa} S (S - n) | \nabla h |^2 \right]d M
  \nonumber\\
  & = & \int_M \left[ 2 (1 + \eta^{- 1}) n \kappa + \frac{1}{8 (1 + \eta^{-
  1}) n \kappa} S (S - n) \right] | \nabla h |^2 d M \label{dh1}\\
  & \leq & \int_M \left[ 2 (1 + \eta^{- 1}) n \kappa + \frac{1}{8
  \kappa} (S - n) \right] | \nabla h |^2 d M. \nonumber
\end{eqnarray}
Combining (\ref{3Bsigma}), (\ref{dh3}) and (\ref{dh1}), we get
\begin{eqnarray*}
  &  & 3 \int_M (A - 2 B)d M\\
  & \leq & \int_M \left\{ (S + 4) | \nabla h |^2 + \frac{1}{3} c \sigma^2
  F \right.\\
  &  & \left. + \frac{2}{3 \sigma} \left[ \left(2 (1 + \eta^{- 1}) n \kappa +
  \frac{1}{8 \kappa} (S - n)\right) | \nabla h |^2 + \varepsilon | \nabla^2 h |^2 +
  \frac{1}{16 \varepsilon} | \nabla S |^2 \right] \right\} d M.
\end{eqnarray*}
This together with (\ref{ddh2}) and (\ref{int2B}) implies
\begin{eqnarray*}
  &  & 3 \int_M (A - 2 B)d M\\
  & \leq & \int_M \left\{ (S + 4) | \nabla h |^2 + \frac{1}{3} c \sigma^2
  \left( 2 (A - 2 B) - 2 | \nabla h |^2 + \frac{1}{2} | \nabla S |^2 \right)
  \right.\\
  &  & + \frac{2}{3 \sigma} \left[ \left(2 (1 + \eta^{- 1}) n \kappa + \frac{1}{8
  \kappa} (S - n)\right) | \nabla h |^2 + \frac{1}{16 \varepsilon} | \nabla S |^2
  \right.\\
  &  & \left. \left. + \varepsilon \left(- (2 n + 3 - S) | \nabla h |^2 + 3 (A - 2
  B) + \frac{3}{2} | \nabla S |^2 \right) \right] \right\} d M.
\end{eqnarray*}
This implies
\begin{eqnarray}
  &  & \int_M [\theta (A - 2 B) - \tau | \nabla S |^2] d M \nonumber\\
  & \leq & \int_M \left[ S + 4 - \frac{2}{3} c \sigma^2 + \frac{2}{3
  \sigma} \times \right.  \label{thetaB}\\
  &  & \left. \left( 2 (1 + \eta^{- 1}) n \kappa + \frac{1}{8 \kappa} (S - n)
  - \varepsilon (2 n + 3 - S) \right) \right] | \nabla h |^2 d M, \nonumber
\end{eqnarray}
where
\[ \theta = 3 - \frac{2}{3} c \sigma^2 - 2 \sigma^{- 1} \varepsilon, \qquad
   \tau = \frac{1}{6} c \sigma^2 + \frac{2}{3 \sigma} \left( \frac{1}{16
   \varepsilon} + \frac{3 \varepsilon}{2} \right) . \]
We restrict $\sigma$ and $\varepsilon$ such that $\theta \geq 0$.

By (\ref{intdS2}), we get
\begin{eqnarray}
\int_M | \nabla S |^2 d M & \leq & 2 \int_M \left[ (n - S) |
\nabla h |^2 +
\frac{n}{\eta} S (S - n) \right] d M \nonumber\\
& = & 2 \int_M \left( n - S + \frac{n}{\eta} \right) | \nabla h |^2
d M. \label{dS2less}
\end{eqnarray}
Combining (\ref{intBgre}), (\ref{thetaB}) and (\ref{dS2less}), we
obtain
\begin{eqnarray}\label{ineq3p}
0 & \leq & \int_M \left[ S + 4 - \frac{2}{3} c \sigma^2 + \frac{2}{3
    \sigma} \times \right. \nonumber\\
&  & \left( 2 (1 + \eta^{- 1}) n \kappa + \frac{1}{8 \kappa} (S - n) -
\varepsilon (2 n + 3 - S) \right) \nonumber\\
&  & \left. + 2 \tau \left( n - S + \frac{n}{\eta} \right) - \theta \left(
1 - \frac{n}{6} + \frac{5}{6} S - \frac{b n}{\eta} \right) \right] | \nabla
h |^2 d M \label{intineq}\\
& = & \int_M \left[ (S - n) \left( 1 + \frac{1}{12 \sigma \kappa} + \frac{2
    \varepsilon}{3 \sigma} - 2 \tau - \frac{5}{6} \theta \right) \right.
\nonumber\\
&  & + n + 4 - \frac{2}{3} c \sigma^2 + \frac{2}{3 \sigma} \Big(2 (1 + \eta^{-
    1}) n \kappa - \varepsilon ( n + 3 )\Big) \nonumber\\
&  & \left. + \frac{2 \tau n}{\eta} - \theta \left( 1 + \frac{2
n}{3} - \frac{b n}{\eta} \right) \right] | \nabla h |^2 d M.
\nonumber
\end{eqnarray}
Taking $\varepsilon = \frac{1}{18}, \sigma = \frac{7}{18}, \kappa =
\frac{1}{24}, \eta = 18$, we get
\[\theta = \frac{784}{513 n}+\frac{6323}{2835} > 0.\]
It follows from (\ref{intineq}) that
\[ 0 \leq \int_M \left[ - (S - n) \left( \frac{784}{1539 n}+\frac{13}{2430}
    \right) - \frac{3629 n^2+126690 n-347760}{1939140
   (n + 4)} \right] | \nabla h |^2 d M. \]
Since $n \geq 6$, the expression in the square bracket of the
above formula is bounded from above by a negative constant. This
implies $| \nabla h | \equiv 0$. Therefore, $S \equiv n$.
\end{proof}

\begin{rmk}
    In fact, we can enlarge the second pinching interval in Main Theorem to $[n, n + \frac{n}{17.93}]$ by taking $\eta = 17.93$.
\end{rmk}

Inspired by Tang-Yan's beautiful work on the famous Yau conjecture
for the first eigenvalue \cite{TY,Y2}, we have the following open
problem.
\begin{problem}\label{prob1} Let $M$ be an $n$-dimensional compact minimal hypersurface in
$\mathbb{S}^{n+1}$. Denote by $\lambda_1(M)$ the first eigenvalue of the Laplace operator acting on functions over $M$.\\
(i) Is it possible to prove that if $M$ has constant scalar curvature, then $\lambda_1(M)=n$?\\
(ii) Set $a_k=(k-sgn(5-k))n$. Is it possible to prove that if
$a_k\leq S\leq a_{k+1}$ for some $k\in \{m\in\mathbb{Z}^+; 2\leq
m\leq 4\,\}$, or $S\geq a_{5}$, then $\lambda_1(M)=n$?
\end{problem}

It is well known that the possible values of the squared length of
the second fundamental forms of all closed isoparametric
hypersurfaces with constant mean curvature $H$ in the unit sphere
form a discrete set $I(\subset\mathbb{R})$, which was explicitly
given by Muto \cite{Mut}. The following problem can be viewed as a
general version of the Chern conjecture.
\begin{problem}\label{prob2} Let $M$ be a closed hypersurface with
constant mean curvature $H$
in the unit sphere $\mathbb{S}^{n+1}$. \\
(i) Assume that $a\leq S \leq b$, where $a<b$ and $[a,b]\bigcap
I=\{a,b\}$. Is it possible to prove that $S\equiv a$ or $S\equiv b,$
and $M$ is an isoparametric
hypersurface in $\mathbb{S}^{n+1}$?\\
(ii) Suppose that $S\geq c,$ where $c=\sup_{t\in I} t.$ Can one show
that $S\equiv c$, and $M$ is an isoparametric hypersurface in
$\mathbb{S}^{n+1}$?
\end{problem}

Some progresses on the generalized Chern conjecture for
hypersurfaces with constant mean curvature have been made by several
authors \cite{AB,C2,ChengWan,LXX,XT,XX1,XX2}, ect.

\end{document}